\documentclass[12pt,a4paper]{amsart}
\usepackage[width=5.65in,height=8.0in,centering]{geometry}
\usepackage{mathtools}
\usepackage{amssymb}
\usepackage{amscd}

\usepackage{enumerate}
\usepackage{enumitem}
\usepackage{etoolbox}
\usepackage{thmtools}
\usepackage{comment}

\usepackage{xcolor}
\usepackage{graphicx}
\usepackage{tikz}
\usetikzlibrary{calc,fit,matrix,positioning}
\usepackage{tikz-cd}

\usepackage[
	colorlinks=true,
	linkcolor=blue,
	citecolor=purple,
	urlcolor=blue,
	linktoc=page,
	plainpages=false,
	pdfpagelabels,
	backref=page
]{hyperref}
\usepackage[capitalize,noabbrev]{cleveref}

\numberwithin{equation}{section}

\DeclarePairedDelimiter{\card}{\lvert}{\rvert}

\newcommand{\mc}[1]{\mathcal{#1}}

\newcommand{\newWord}[1]{\emph{#1}}
\newcommand{\angl}[1]{\left\langle #1 \right\rangle}    

\newcommand{\KK}{\Bbbk}

\DeclareMathOperator{\ann}{\mathrm{ann}}
\DeclareMathOperator{\cl}{cl}

\DeclareMathOperator{\HF}{HF}

\DeclareMathOperator{\HS}{HS}

\DeclareMathOperator{\pdim}{pdim}
\DeclareMathOperator{\reg}{reg}

\newcommand{\GMA}[1]{\mathsf{GMA}(#1)}

\newcommand{\M}{\mathsf{M}}
\newcommand{\U}{\mathsf{U}}
\DeclareMathOperator{\rk}{rk}

\theoremstyle{plain}
\newtheorem{theorem}{Theorem}[section]
\newtheorem{conjecture}[theorem]{Conjecture}
\newtheorem{lemma}[theorem]{Lemma}
\newtheorem{proposition}[theorem]{Proposition}
\newtheorem{question}[theorem]{Question}

\theoremstyle{definition}

\newtheorem{example}[theorem]{Example}
\newtheorem{notation}[theorem]{Notation}

\theoremstyle{remark}
\newtheorem{remark}[theorem]{Remark}

\title[Betti numbers of graded M\"obius algebras]{Graded Betti numbers of graded M\"obius algebras of uniform matroids}

\author[]{Louiza Fouli}
\address{New Mexico State University, Department of Mathematics, Las Cruces, NM, USA}
\email{lfouli@nmsu.edu}
\author[]{Sean Grate}
\address{Iowa State University, Department of Mathematics, Ames, IA, USA}
\email{sgrate@iastate.edu}
\author[]{Selvi Kara}
\address{Bryn Mawr College, Department of Mathematics, Bryn Mawr, PA, USA}
\email{skara@brynmawr.edu}
\author[]{Adam LaClair}
\address{University of Nebraska-Lincoln, Department of Mathematics, Lincoln, NE, USA}
\email{alaclair2@unl.edu}
\author[]{Jason McCullough}
\address{Iowa State University, Department of Mathematics, Ames, IA, USA}
\email{jmccullo@iastate.edu}
\author[]{Vinh Nguyen}
\address{New Mexico State University, Department of Mathematics, Las Cruces, NM, USA}
\email{vnguyen@nmsu.edu}
\author[]{Aleksandra Sobieska}
\address{Marshall Univeristy, Department of Mathematics \& Physics, Huntington, WV, USA}
\email{sobieskasnyd@marshall.edu}
\author[]{Prajwal Udanshive}
\address{Western University, Department of Mathematics, London, ON, Canada}
\email{pudanshi@uwo.ca}

\begin{document}
\begin{abstract}
Graded M\"obius algebras were a key tool in the proof of the Dowling-Wilson Top Heavy Conjecture.  They are commutative algebras whose Hilbert functions recover the Whitney numbers of the second kind, i.e. the number of flats of a given rank.  The graded Betti numbers of the defining ideal of a graded M\"obius algebra refine the Hilbert function and describe its minimal free resolution.  In this paper we derive precise formulas for the graded Betti numbers of the defining ideals of graded M\"obius algebra for any uniform matroid.  We also study when the graded M\"obius algebra of an arbitrary matroid is linearly presented.
\end{abstract}



\maketitle

\section{Introduction}

The graded M\"obius algebra of a matroid $\M$, denoted $\GMA{\M}$, is a commutative algebra whose Hilbert function counts the number of flats of $\M$ of a given rank, called the \newWord{Whitney numbers of the second kind} of the lattice of flats of $\M$ (see ~\cite{Oxley}).  It was used in the proof of the Dowling--Wilson Top Heavy Conjecture \cite{singular:Hodge:theory:for:combinatorial:geometries}.  The algebraic properties of the graded M\"obius algebras were studied in LaClair et. al. \cite{LMMP24}.  In particular, they constructed a presentation of the graded M\"obius algebra as well as a universal Gr\"obner basis for its defining ideal.  In the graphic setting, they also gave a complete characterization of which graded M\"obius algebras are Koszul.  In many ways, graded M\"obius algebras share similarities with the more well-studied Orlik-Solomon algebras. We refer to the survey~\cite{yuz01} as a starting point for the latter. 

The main purpose of this paper is to investigate the graded Betti table of the defining ideal of the graded M\"obius algebra of a matroid over the defining polynomial ring.  In this note, we give a complete description of the graded Betti tables in the case of uniform matroids; see \Cref{thm:betti-linkage}.

The primary technique to compute the graded Betti numbers is via linkage theory.  We first compute the Betti table of an ideal directly linked to the defining ideal of the graded M\"obius algebra and  use this information about the Hilbert function to compute all of the graded Betti numbers.

For arbitrary matroids, the structure of the graded Betti table of the defining ideal $I_\M$ of the graded M\"obius algebra is not yet clear.  It was shown in \cite{LMMP24} that if $\M$ is chordal, then $\GMA{\M}$ is quadratic, meaning that $I_\M$ is generated by elements of degree at most $2$.  We conjecture (see \Cref{conj:Sylv_iff_lin_pres}) that for matroids with at least 3 elements for which $\GMA{\M}$ is quadratic, $I_\M$ has only linear first syzygies, i.e. $\GMA{\M}$ is linearly presented, if and only if $\M$ is Sylvester, meaning every two elements of $\M$ belong to a unique circuit of size $3$.  We prove the forward implication of our conjecture in \Cref{prop:linpres:Sylvester}.  We also use this to answer a special case of an open question regarding the asymptotic regularity of quadratic linearly presented ideals; see \Cref{prop:lin:pres:reg}.

In Section~\ref{back}, we recall some basics on free resolutions, linkage, matroids, and graded M\"obius algebras.  Section~\ref{main} contains the proof of our main result.  Section~\ref{linpres} contains our results on linearly presented graded M\"obius algebras.

\section{Background}\label{back}

In this section we collect terminology and a few background results needed for the remainder of the paper.

\subsection{Free Resolutions}

Let $\KK$ denote a field, and let $S = \KK[x_1,\ldots,x_n]$ be a standard graded polynomial ring over $\KK$.  Let $\mathfrak{m}$ denote the graded maximal ideal of $S$, and write $S(-j)$ for the free rank-one $S$-module with grading $S(-j)_i = S_{i-j}$.  Given a finitely generated graded $S$-module $A = \bigoplus_{i} A_i$, its \newWord{Hilbert function} is $\HF_A(i) = \dim_{\KK}(A_i)$ and its \newWord{Hilbert series} is $\HS_A(t) = \sum_{i \ge 0} \HF_A(i) t^i$.  
A \emph{minimal graded free resolution} $F_\bullet$ of the module $A$ is a complex of graded free modules $F_i = \oplus_j S(-j)^{\beta_{i,j}(A)}$ such that
\[0 \to F_p \xrightarrow{\partial_p} F_{p-1} \xrightarrow{\partial_{p-1}} \cdots \xrightarrow{\partial_2} F_1 \xrightarrow{\partial_1} F_0 \to A \to 0\]
is exact and $\mathrm{Im}(\partial_i) \subseteq \mathfrak{m} F_{i-1}$. It is unique up to isomorphism of complexes and thus the \newWord{graded Betti numbers} $\beta_{i,j}(A)$ are invariants of $A$.  Since the Hilbert series of $S$ is $\HS_S(t) = \frac{1}{(1-t)^n}$, the Hilbert series of $A$ is determined by the graded Betti numbers via 
\[
\HS_A(t) = \sum_{i \ge 0} \frac{(-1)^i \beta_{i,j}(A) t^j}{(1-t)^n}.
\]
Thus, the graded Betti numbers are a finer set of invariants than the Hilbert series.  The graded Betti numbers are best displayed in the \newWord{graded Betti table} of $A$ in which $\beta_{i,j}(A)$ appears in column $i$ and row $j-i$. The \newWord{Castelnuovo-Mumford regularity} $A$, denoted $\reg_S(A)$, is defined as $\max\{j-i \mid \beta_{i,j}^S(A) \neq 0\}$, while the projective dimension of $A$, denoted $\mathrm{pd}_S(A)$, is defined as $\max\{i \mid \beta_{i,j}^S(A) \neq 0\}$. When $A$ is Artinian, it is well-known that $\reg_S(A) = \deg \HS_A(t)$ and $\mathrm{pd}_S(A) = n$. For related notation and definitions we refer the reader to \cite{graded:syzygies}.

\subsection{Linkage}

Again let $S = \KK[x_1,\ldots,x_n]$ be a standard graded polynomial ring over a field $\KK$.  An ideal $D \subseteq S$ is a \textit{complete intersection} if it is generated by a regular sequence, that is, $D = \angl{f_1,\ldots,f_c}$ and $f_i$ is a nonzerodivisor on $S/\angl{f_1,\ldots,f_{i-1}}$ for all $i$.  Two ideals $I,J \subseteq S$ are \textit{directly linked by $D$} if there exists a complete intersection ideal $D\subseteq I\cap J$ such that $I=D:J$ and $J=I:D$.
If this is the case, both $I$ and $J$ are necessarily unmixed. All of the ideals we consider will be graded ideals.  In this case, we can translate information about the free resolutions and Hilbert functions of an ideal to one directly linked to it.  

    Let $I, J \subseteq S$ be graded ideals directly linked by a graded complete intersection $D$ generated by homogeneous elements $f_1,\ldots,f_c$ with $\deg(f_i) = d_i$.  Set $d = d_1 + d_2 + \cdots + d_c$.  Then
     $I/D \cong \omega_{S/J}(-c)$, where $\omega_{S/J} = \mathrm{Ext}^c_S(S/J,S)(-c)$ denotes the canonical module of $S/J$.
       Using the short exact sequence
        \[
        0 \to \omega_{S/J}(-c) \to S/D \to S/I \to 0, 
        \]
        one can relate the graded Betti numbers and Hilbert series of $S/J$ and $S/I$. 

\subsection{Matroids and Lattices}

In this subsection, we recall the basics of matroids necessary for our use in this paper. We refer to~\cite{Oxley} for other undefined terms from matroid theory.

A matroid abstracts the concept of independence. Specifically, a \newWord{matroid} $\M$ is a pair $(E,\mc{I})$, where $E$ (called the \newWord{ground set}) is a finite set, and $\mc{I}$ (called the \newWord{independent sets}) consists of subsets of $E$ satisfying the following conditions:
\begin{enumerate}
    \item $\varnothing \in \mc{I}$,
    \item If $A \subseteq B$ and $B \in \mc{I}$, then $A \in \mc{I}$,
    \item If $A, B \in \mc{I}$, and $\card{A} < \card{B}$, then there exists $e \in B \smallsetminus A$ such that $A \cup \{e\} \in \mc{I}$.
\end{enumerate}
The first two conditions imply that $\mc{I}$ is an abstract simplicial complex. The third condition is called \newWord{the augmentation axiom}, generalizing the replacement theorem in linear algebra. 

Given $X \subseteq E$, we define the \newWord{rank} of $X$, denoted $\rk_{\M}(X)$, as the cardinality of the largest independent set contained in $X$. A \emph{basis} of $\M$ is a maximal independent set, so $\rk_{\M}(E) = |B|$ for any basis $B$ of $\M$.  A set that is not independent is called \emph{dependent} and a minimal dependent set is called a \emph{circuit} of $\M$.

For $X \subseteq E$, we define the \newWord{closure} of $X$, denoted $\cl(X)$, as $\cl(X) := \{x \in e \mid \rk(X \cup \{e\}) = \rk(X)\}$. A \newWord{flat} $F$ of $\M$ is a subset of $E$ with $\cl(F) = F$. The set of all flats forms a lattice under inclusion, which is called the \newWord{lattice of flats}, and denoted by $\mc{L}(\M)$. The \newWord{meet} of two flats $F$ and $F'$, denoted $F\wedge F'$, is the largest flat contained in both $F$ and $F'$ and is given by $F \cap F'$. The \newWord{join} of two flats $F$ and $F'$, denoted $F\vee F'$, is the smallest flat containing both $F$ and $F'$ and is given by $\cl(F \cup F')$.

\def\hdist{1.75}
\begin{example}
The \newWord{uniform matroid of rank $r$ on $n$ elements}, denoted $\U_{r,n}$ is the matroid having $E = \{1,2,\ldots,n\} =[n]$ and $\mc{I} := \{A \subseteq E \mid \card{A} \leq r\}$. For example, the lattice of flats of $\U_{3,5}$ is displayed in Figure~\ref{uniformlattice}.
\begin{figure}[htb!]
\begin{center}
\begin{tikzpicture}[scale=0.8]

\node (bot) at (0,0) {$\varnothing$};

\node (s1) at (-4,2) {$1$};
\node (s2) at (-2,2) {$2$};
\node (s3) at (0,2) {$3$};
\node (s4) at (2,2) {$4$};
\node (s5) at (4,2) {$5$};

\node (p12) at (-4.5*\hdist,4) {$12$};
\node (p13) at (-3.5*\hdist,4) {$13$};
\node (p14) at (-2.5*\hdist,4) {$14$};
\node (p15) at (-1.5*\hdist,4) {$15$};
\node (p23) at (-.5*\hdist,4) {$23$};
\node (p24) at ( 0.5*\hdist,4) {$24$};
\node (p25) at ( 1.5*\hdist,4) {$25$};
\node (p34) at ( 2.5*\hdist,4) {$34$};
\node (p35) at ( 3.5*\hdist,4) {$35$};
\node (p45) at ( 4.5*\hdist,4) {$45$};

\node (t123) at (-4.5*\hdist,6) {$123$};
\node (t124) at (-3.5*\hdist,6) {$124$};
\node (t125) at (-2.5*\hdist,6) {$125$};
\node (t134) at (-1.5*\hdist,6) {$134$};
\node (t135) at (-0.5*\hdist,6) {$135$};
\node (t145) at ( 0.5*\hdist,6) {$145$};
\node (t234) at ( 1.5*\hdist,6) {$234$};
\node (t235) at ( 2.5*\hdist,6) {$235$};
\node (t245) at ( 3.5*\hdist,6) {$245$};
\node (t345) at ( 4.5*\hdist,6) {$345$};

\node (top) at (0,8) {$12345$};

\foreach \i in {1,2,3,4,5}
  \draw (bot)--(s\i);

\draw (s1)--(p12); \draw (s2)--(p12);
\draw (s1)--(p13); \draw (s3)--(p13);
\draw (s1)--(p14); \draw (s4)--(p14);
\draw (s1)--(p15); \draw (s5)--(p15);

\draw (s2)--(p23); \draw (s3)--(p23);
\draw (s2)--(p24); \draw (s4)--(p24);
\draw (s2)--(p25); \draw (s5)--(p25);

\draw (s3)--(p34); \draw (s4)--(p34);
\draw (s3)--(p35); \draw (s5)--(p35);

\draw (s4)--(p45); \draw (s5)--(p45);

\draw (p12)--(t123); \draw (p13)--(t123); \draw (p23)--(t123);

\draw (p12)--(t124); \draw (p14)--(t124); \draw (p24)--(t124);

\draw (p12)--(t125); \draw (p15)--(t125); \draw (p25)--(t125);

\draw (p13)--(t134); \draw (p14)--(t134); \draw (p34)--(t134);

\draw (p13)--(t135); \draw (p15)--(t135); \draw (p35)--(t135);

\draw (p14)--(t145); \draw (p15)--(t145); \draw (p45)--(t145);

\draw (p23)--(t234); \draw (p24)--(t234); \draw (p34)--(t234);

\draw (p23)--(t235); \draw (p25)--(t235); \draw (p35)--(t235);

\draw (p24)--(t245); \draw (p25)--(t245); \draw (p45)--(t245);

\draw (p34)--(t345); \draw (p35)--(t345); \draw (p45)--(t345);

\foreach \t in {t123,t124,t125,t134,t135,t145,t234,t235,t245,t345}
  \draw (\t)--(top);

\end{tikzpicture}
\end{center}
\caption{The lattice of flats of $\U_{3,5}$.}
\label{uniformlattice}
\end{figure}
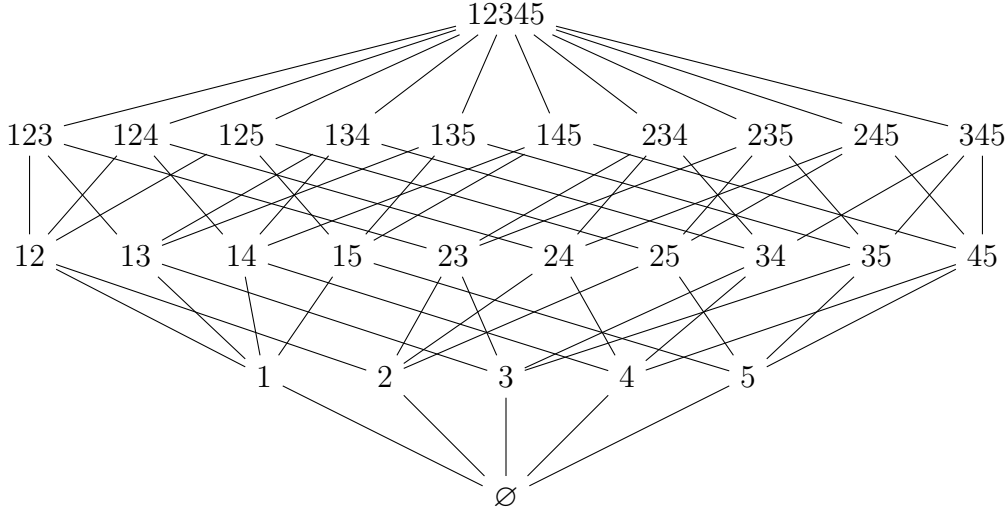

\end{example}

\subsection{Graded M\"obius Algebras}

 Let $\KK$ be a field of characteristic $0$.
 Let $\M$ be a simple matroid on $[n]:=\{1,\dots,n\}$, and let $\mathcal{L} = \mathcal{L}(\M)$ be its lattice of flats.  The \textit{graded M\"obius algebra} of $\M$ (over $\KK$), denoted $\GMA{\M}$, is the graded $\KK$-vector space
 \[
\GMA{\M} = \bigoplus_{F \in \mathcal{L}} \KK y_F,
 \]
 with (commutative) multiplication defined by 
 \[
y_F y_G = \begin{cases}
    y_{F \vee G} & \text{ if } \rk(F) + \rk(G) = \rk(F \vee G),\\
    0 & \text{ otherwise. }
\end{cases}
 \]
Because $\mathcal{L}$ is atomic, $\GMA{\M}$ is a standard graded $\KK$-algebra and admits a presentation as a quotient of a standard graded polynomial ring.  Now set
$S\coloneqq \KK[y_1,\dots,y_n]$ with $\deg(y_i)=1$.
The \emph{graded M\"obius algebra} of $\M$ is the Artinian graded algebra
\[
\GMA{\M}\coloneqq S/I_\M,
\]
where $I_\M$ is the defining ideal. 

The graded M\"obius algebra notably appeared in \cite{singular:Hodge:theory:for:combinatorial:geometries} and \cite{semi-small:decompositions} as an algebraic tool to prove the Dowling--Wilson Top Heavy Conjecture. The Hilbert function of the graded M\"obius algebra counts the numbers of flats of the matroid of a fixed rank, i.e.\ the Whitney numbers of the second kind. In \cite{Maeno:Numata:published}, Maeno and Numata investigated algebraic properties of graded M\"obius algebras, proving that $\GMA{\M}$ is Gorenstein if and only if $\mathcal{L}(\M)$ is modular. In \cite{LMMP24}, the authors advanced the study of the algebraic properties of the graded M\"obius algebra, describing a Gr\"obner basis for $I_\M$ and characterizing graphic matroids realizing quadratic and Koszul graded M\"obius algberas.

It is shown in \cite[Proposition 3.1]{LMMP24} that
$I_\M$ admits a Gr\"obner basis of the form $D+L$, where
\[
D=\angl{y_i^2 \mid 1 \leq i\leq n}
\qquad\text{and}\qquad
L=\angl{y_{C\setminus i}-y_{C\setminus j} \mid C\ \text{a circuit of }\M,\ i,j\in C},
\]
where for a subset $T \subseteq [n]$, write $y_T\coloneqq \prod_{i\in T}y_i$.

  In this paper, we continue to investigate the algebraic properties of the graded M\"obius algebra, explicitly describing the $S$-graded Betti numbers of $\GMA{\M}$ for every uniform matroid $\M$.


\section{Main Result}\label{main}

In this section we specialize to the uniform matroid $\U_{r,n}$ (so every circuit has size $r+1$) and
use linkage with the complete intersection $D$ above to relate the graded Betti numbers of $S/I_\M$ to those of the
link $J\coloneqq D:I_\M$.  If $r = n$, then there are no circuits, and $I_\M = D$ is merely a complete intersection of quadrics which is resolved by a Koszul complex.  If $r = 0$, then $I_\M = \mathfrak{m}$, which is resolved by a linear Koszul complex.  If $r = 1$, then $I_\M = \angl{x_n^2,\ x_i-x_n\mid 1\leq i<n}$, which is again a complete intersection.  Thus we may assume for the remainder of this section, except where noted, that $2 \le r \le n-1$.

\begin{remark}[Linkage and the canonical module]\label{rem:linkage}
Since $D$ is a complete intersection of quadrics, its Hilbert series is $\HS_{S/D}(t) = (1+t)^n$, and its nonzero graded Betti numbers are
\[
\beta_{i,2i}(S/D)=\binom{n}{i}\qquad (0\le i\le n).
\]
Moreover, $I_\M$ is Artinian, so in particular unmixed, and $D\subseteq I_\M$, so we may form the link $J\coloneqq D:I_\M$.
Classical liaison theory yields a short exact sequence
\begin{equation}\label{eq:liaison-ses}
0 \longrightarrow \omega_{S/J}(-n) \longrightarrow S/D \longrightarrow S/I_\M \longrightarrow 0,
\end{equation}
where $\omega_{S/J}$ denotes the canonical module of $S/J$.
Taking Hilbert series in \eqref{eq:liaison-ses} gives
\begin{equation}\label{eq:HS-liaison}
\HS_{S/D}(t) = \HS_{S/I_\M}(t) + t^n\, \HS_{S/J}\!\left(\frac{1}{t}\right).
\end{equation}
Finally, we use the  Betti-duality for an Artinian quotient:
\begin{equation}\label{eq:betti-canonical}
\beta_{i,j}(\omega_{S/J})=\beta_{n-i,n-j}(S/J).
\end{equation}
\end{remark}

First we record the Hilbert series of the graded M\"obius algebra for a uniform matroid.

\begin{lemma}\label{HS1} Let $n, r$ be integers with $0\le r\le n$ and let $\M=\U_{r,n}$ be the uniform matroid of rank $r$ on $[n]$. Then 
    $$\HS_{\GMA{\M}}(t) = \sum_{i = 0}^{r-1} \binom{n}{i} t^i + t^r.$$
\end{lemma}

\begin{proof}
    This follows since $\dim_{\KK}(\GMA{\M})$ counts the number of flats of rank $i$ in $\mathcal{L}(\U_{r,n})$.  For $i < r$, every subset of size $i$ is a flat.  The ground set is the only flat of rank $r$, and there are no flats of higher rank.
\end{proof}

This immediately allows us to compute the Hilbert series of the linked ideal $J$.

\begin{lemma}\label{HSJ}
    With the setup as in \Cref{rem:linkage}, 
    $$\HS_{S/J}(t) = \sum_{i=0}^{n-r-1}\binom{n}{i}t^i + \left(\binom{n}{r}-1\right)t^{n-r}.$$
    In other words,
    $$\HF_{S/J}(i) = \begin{cases}
        \binom{n}{i}\  & \text{ if } 0\le i\le n-r-1,\\
       \binom{n}{r}-1 & \text{ if } i = n-r,\\
       0 & \text{ otherwise}.
    \end{cases}$$
\end{lemma}

\begin{proof}
    By \Cref{HS1}, $\HS_{\GMA{\M}}(t) = \sum_{i = 0}^{r-1} \binom{n}{i} t^i + t^r.$  So 
\begin{align*}
\HS_{S/J}(t)
&=(1+t)^n - t^n\, \HS_{\GMA{\M}}\!\left(\frac{1}{t}\right)\nonumber\\
&=\sum_{i=0}^{n-r}\binom{n}{i}t^i - t^{n-r} \notag \\
&=\sum_{i=0}^{n-r-1}\binom{n}{i}t^i + \left(\binom{n}{r}-1\right)t^{n-r}. \qedhere
\end{align*}
\end{proof}

We now describe a generating set of the ideal $L$, where $I_\M = D + L$. 

\begin{lemma}\label{obs:L-min} Let $\M = \U_{r,n}$. Fix the $r$-subset $Z\coloneqq \{n-r+1,\dots,n\}\subseteq [n]$. Then
\[
L=\angl{y_W-y_Z:\ W\subseteq [n],\ |W|=r,\ W\neq Z}.
\]
\end{lemma}

\begin{proof}
Pick an arbitrary element $y_{C\setminus i} - y_{C\setminus j}\in L$, where $C$ is a circuit with $i,j\in C$; rewrite it as
\[
y_{C\setminus i}-y_{C\setminus j}=(y_{C\setminus i}-y_Z)-(y_{C\setminus j}-y_Z).
\] This yields one containment.

Conversely, let $V,W$ be $r$-subsets of $[n]$, and let $V=V_0,V_1,\dots,V_k=W$ in which each consecutive pair $(V_{\ell}, V_{\ell+1})$ satisfies
$|V_\ell\triangle V_{\ell+1}|=2$.
 Thus, being an $(r+1)$-element set, $V_\ell\cup V_{\ell+1}$ is a circuit of  $\U_{r,n}$, and hence
$y_{V_\ell}-y_{V_{\ell+1}}\in L$.
Summing along the sequence yields $y_V-y_W\in L$, and taking $W=Z$ recovers the claim.
\end{proof}

 We adopt the following notation for the rest of this section. 

\begin{notation}
Let $d=n-r$.  Write $\mathfrak{m} = (y_1,\ldots,y_n)$ and let $\mathfrak{m}^{[d+1]}\subseteq S$ be the ideal generated by all squarefree monomials of degree $d+1$.
Set 
\[
\ell_d \coloneqq \sum_{\substack{T\subseteq [n]\\ |T|=d}} y_T.
\]
\end{notation}

Next, we turn our attention to the minimal generating set of the link  $J=D:I_\M$. 

\begin{lemma}\label{lem:J-gens}
For $\M=\U_{r,n}$, the link $J=D:I_\M$ is generated by
\[
J = D + \mathfrak{m}^{[d+1]} + (\ell_d).
\]
Moreover, $J$ has exactly $\binom{n}{d+1} - n$ minimal squarefree generators of degree $d+1$.
\end{lemma}

\begin{proof}
Since $I_\M=D+L$, we have $J=D:L$.  As $D \subseteq J$, describing the minimal generators of $J$ in $S$ is equivalent to computing $\ann_A(L)= 0:_A L$, where $A := S/D$.  
In $A$, every nonzero monomial is squarefree. So,  for squarefree monomials $y_V$ and $y_W$, we have
$y_Vy_W=0$ if and only if $V\cap W\neq\varnothing$.

Let $f \in \ann_A(L)$ be homogeneous with leading squarefree monomial $y_W$.
If $\deg(f) \le d-1$, then choose the lexicographically first $r$-subset $V\subseteq  W^c\coloneq [n]\setminus W$.  Note that since $|W| \le d-1$, $V \neq Z$ where $Z = \{n-r+1,\dots,n\}$ as in \cref{obs:L-min}.
Then $V\cap W=\varnothing$, so $y_W(y_V-y_Z)\neq 0$ in $A$.  As $y_W$ is the leading term of $f$, it follows that $f(y_V - y_Z) \neq 0$ in $A$, a contradiction.
Hence $\deg(f)\ge d$.

If $|W|\ge d+1$, then for every $r$-subset $V$ we have $V\cap W\neq\varnothing$ (since $|V|+|W|>n$),
so $y_W(y_V-y_Z)=0$ in $A$.  Thus every squarefree monomial of degree $d+1$ lies in $\ann_A(L)$.

It remains to understand $\ann_A(L)_d$.
Let
\[
g=\sum_{\substack{V\subseteq [n]\\ |V|=d}} a_V\,y_V\in A_d
\]
where $a_V \in \KK$. Assume \(g\in \ann_A(L)\). Fix an $r$-subset $W\subseteq [n]$. The product $y_Vy_W$ is nonzero in $A$ exactly when $V=W^c$. So, we have the following in $A$:
\[
g\cdot y_W = a_{W^c}\, y_{[n]}.
\]
Since $g(y_W-y_Z)=0$ for all $W$, we obtain $a_{W^c}=a_{Z^c}$ for every $W$. This means all the coefficients $a_V$ are equal.
Therefore $\ann_A(L)_d$ is one-dimensional and generated by $\ell_d$.
Lifting back to $S$, we conclude that $J=D+\mathfrak{m}^{[d+1]}+(\ell_d)$.  That these are the no other generators of $J$ follows from the Hilbert series of $S/J$ given in Lemma~\ref{HSJ}.

Finally, note that multiplying $\ell_d$ by each variable $y_i$ produces $n$ linear relations on the squarefree monomials of degree $d+1$.  Thus there are precisely $\binom{n}{d+1} - n$ minimal generators of degree $d+1$.
\end{proof}

\begin{figure}[htb!]
\centering
\begin{tikzpicture}[every node/.style={font=\small}]

\def\w{1.6}
\def\h{1.1}

\def\C{8}
\def\R{6}

\newcommand{\cellcenter}[2]{({(#1+0.5)*\w},{-(#2+0.5)*\h})}
\newcommand{\cellrect}[2]{({#1*\w},{-(#2)*\h}) rectangle ({(#1+1)*\w},{-(#2+1)*\h})}


\foreach \r in {0,1,2,3}{%
  \foreach \c in {0,1,2,3,4,5,6,7}{%
    \pgfmathtruncatemacro{\diag}{\r}%
    \ifnum\c=\diag\relax
    \else
      \fill[gray!15] \cellrect{\c}{\r};
    \fi
  }%
}

\foreach \c in {0,3,5,6,7}{%
  \fill[gray!15] \cellrect{\c}{4};
}

\draw[thick] (0,0) rectangle (\C*\w,-\R*\h);
\foreach \c in {1,...,7} { \draw (\c*\w,0) -- (\c*\w,-\R*\h); }
\foreach \r in {1,...,5} { \draw (0,-\r*\h) -- (\C*\w,-\r*\h); }

\foreach \c/\lab in {
  0/0,
  1/1,
  2/2,
  3/$\cdots$,
  4/$d\!-\!1$,
  5/$d$,
  6/$\cdots$,
  7/$n$
}{
  \node at ({(\c+0.5)*\w}, {0.55*\h}) {\lab};
}

\foreach \r/\lab in {
  0/0,
  1/1,
  2/2,
  3/$\vdots$,
  4/$d\!-\!1$,
  5/$d$
}{
  \node[anchor=east] at ({-0.25*\w}, {-(\r+0.5)*\h}) {\lab};
}


\node at \cellcenter{0}{0} {$1$};
\node at \cellcenter{1}{1} {$\binom{n}{1}$};
\node at \cellcenter{2}{2} {$\binom{n}{2}$};
\node at \cellcenter{4}{4} {$\binom{n}{d-1}$};

\node at \cellcenter{3}{0} {$\cdots$};
\node at \cellcenter{3}{1} {$\cdots$};
\node at \cellcenter{3}{2} {$\cdots$};
\node at \cellcenter{3}{3} {$\ddots$};
\node at \cellcenter{3}{5} {$\cdots$};

\node at \cellcenter{1}{4} {$1$};

\node at \cellcenter{2}{4} {\Large ?};

\node at \cellcenter{3}{4} {$\cdots$};
\node at \cellcenter{6}{0} {$\cdots$};
\node at \cellcenter{6}{1} {$\cdots$};
\node at \cellcenter{6}{2} {$\cdots$};
\node at \cellcenter{6}{3} {$\cdots$};
\node at \cellcenter{6}{4} {$\cdots$};
\node at \cellcenter{6}{5} {$\cdots$};

\node at \cellcenter{1}{5} {$\binom{n}{r-1}-n$};

\node at \cellcenter{7}{5} {$\binom{n}{r}-1$};

\draw[rounded corners, very thick]
  ({2*\w+0.08*\w},{-(4)*\h-0.08*\h})
  rectangle
  ({3*\w-0.08*\w},{-(5)*\h+0.08*\h});
\node[anchor=west] at ({3.05*\w},{-(4.5)*\h}) {};

\end{tikzpicture}
\caption{Betti table of $S/J$ where $d=n-r$}
\label{fig:bettiTable}
\end{figure}

\begin{remark}
We record the shape of the Betti table of $S/J$ using the Macaulay2 convention (columns correspond to the homological degree $i$,
rows correspond to the shifts $j-i$).  

Since $S/J$ is Artinian, we know that $\pdim (S/J)=n$. Additionally, we have $\reg (S/J) = \deg \HS_{S/J}(t) =n-r$. Moreover,  the degrees  of the minimal generators in \cref{lem:J-gens} force most entries to vanish (zero entries are shaded in \Cref{fig:bettiTable}). The only potential ``new'' entry may happen  in column 2 and row $d-1$, i.e.\ $\beta_{2,d+1} (S/J)$.  If this Betti number is non-zero, then there is a first syzygy of degree $d+1$. Such a syzygy would arise  from a relation
between the complete intersection $D$ and the generator $\ell_d$.

\end{remark}

The following proposition rules out additional degree-$(d+1)$ syzygies.

\begin{proposition}\label{prop:noC-ell-syzygy} Let $n, r$ be integers with $2\le r\le n-1$ and let $\M=\U_{r,n}$.  Then,
 $\beta_{2,d+1} (S/J)=0$.

\end{proposition}

\begin{proof}
If $\beta_{2,d+1} (S/J)\neq 0$, there is a nontrivial homogeneous relation of degree $d+1$. Such a relation can involve only the quadrics in $D$ and the element
$\ell_d$. 
Let
\begin{equation}\label{eq:no_syzygy}
\sum_{i=1}^n a_i\,y_i^2 + b\,\ell_d = 0 
\end{equation}
with $\deg(a_i)=d-1$ and $b=\sum_{i=1}^n b_i y_i$ with $b_i\in\KK$. We claim that $b=0$.

Since $\sum_i a_i y_i^2$ has no squarefree terms, the squarefree part of the left-hand side of \eqref{eq:no_syzygy}  must be equal to zero. After expanding $b\ell_d$ in \eqref{eq:no_syzygy} and grouping all the squarefree terms, one has
\[
\sum_{\substack{W\subseteq [n]\\ |W|=d+1}}\Big(\sum_{i\in W} b_i\Big)y_W=0.
\]

This means
$\sum_{i\in W} b_i=0$ for every $(d+1)$-subset $W$ of $[n]$. Choose subsets $P,Q\subseteq [n]$ of size $(d+1)$ with $P=T\cup\{p\}$ and $Q=T\cup\{q\}$ for some $d$-subset $T$ and
distinct $p,q\notin T$.
Subtracting the two equations gives $b_p=b_q$.
Varying $p,q$ shows $b_1=\cdots=b_n=c$ for some $c\in\KK$.
Taking any $(d+1)$-subset $W$ then yields $(d+1)c=0$. Hence $c=0$ and $b=0$.
\end{proof}

\begin{remark}\label{rem:betti_shape}
    To summarize, the graded Betti table of $S/J$ has the following shape:
\begin{itemize}
\item for $0\le i\le d-2$, the only nonzero entry in row $i$ is $\beta_{i,2i}(S/J)=\binom{n}{i}$;
\item in row $d-1$, the only nonzero entries are $\beta_{1,d}(S/J)=1$ and $\beta_{d-1,2d-2}(S/J)=\binom{n}{d-1}$, unless $d - 1 = 1$ in which case $\beta_{1,d}(S/J) = \beta_{1,2}(S/J) = \binom{n}{1} + 1 = n+1$.
\item all remaining nonzero graded Betti numbers lie in the last row $d$. 
\end{itemize}
\end{remark}

We now focus on expressing the graded Betti numbers of the last row of the Betti table of $S/J$. As we see in the following proposition, these Betti numbers are determined by the Hilbert series.

\begin{proposition}\label{prop:betti-from-HS}


For each $0\le i\le n$, define 
\begin{equation*}
B_i
:=
 \Bigg(\binom{n}{r}-1 \Bigg)\binom{n}{i}
\;+\;
\sum_{j=1}^{n-r}(-1)^j\binom{n}{r+j}\binom{n}{i+j}.
\end{equation*}

Then the last-row Betti numbers $\beta_{i,i+d}(S/J)$ are given as follows:

\begin{enumerate}[label=\textup(\alph*\textup)]
\item  When $d-1\le i\le n$, we have  
\[
\beta_{i,i+d}(S/J)=B_i.
\]



\item  Let  $1\le i\le d-2$. If $i+d$ is \emph{odd}, then 
\[
\beta_{i,i+d}(S/J)=B_i.
\]
If $i+d$ is \emph{even}, set
\(a=\frac{d-i}{2}\). Then
\[
\beta_{i,i+d}(S/J)=B_i-(-1)^a\binom{n}{r+a}.
\]

\end{enumerate}
\end{proposition}

\begin{proof}
Consider the Poincar\'e polynomial of $S/J$: 
\begin{equation}\label{eq:p-poly}
P(t)=(1-t)^n\HS_{S/J}(t)=\sum_{m\ge 0}p_m t^m.
\end{equation}
Since the Hilbert series can be expressed in terms of the graded Betti numbers, we also have 
\[
P(t)=\sum_{i=0}^n\sum_{j\ge 0}(-1)^i\beta_{i,j}(S/J)\,t^j.
\]
These two expressions of $P(t)$ allow us to connect the Hilbert series  with the Betti numbers. First notice that, for each internal degree $m$, we have
\[
p_{m}=\sum_{u=0}^n(-1)^u\,\beta_{u,m}(S/J).
\]
So, $p_m$ is the alternating sum of all Betti numbers in internal degree $m$.  By the shape of the Betti table discussed in \Cref{rem:betti_shape}, there are  only one or two nonzero Betti numbers contributing to each $p_{m}$.

\begin{itemize}
\item In row $d-1$, there are two non-zero entries:  $\beta_{1,d}=1$ and  $\beta_{d-1,2d-2}=\binom{n}{r-1}$. Notice that $\beta_{1,d}=1$    contributes  to $p_d$ and its contribution is $(-1)^1\cdot 1=-1$. Similarly, $\beta_{d-1,2d-2}=\binom{n}{r-1}$   contributes  to $p_{2d-2}$ and its contribution is $(-1)^{d-1}\binom{n}{r-1}$.
\item The \newWord{Koszul diagonal} entries $\beta_{u,2u}=\binom{n}{u}$ for $0\le u\le d-2$
contribute $(-1)^u\binom{n}{u} $ to $p_{2u}$.
\end{itemize}
The remaining contributions to $p_m$
reduce to the last row of the Betti table and this row consists of $\beta_{i,i+d} (S/J)$ for $1\leq i\leq n$. The contribution of each entry to $p_{i+d}$ is  $(-1)^i\beta_{i,i+d}$.

 Our goal is to express the Betti numbers in the last row in terms of $p_{i+d}$. Now, fix $i$. By plugging in the Hilbert series of $S/J$ from \cref{HSJ}  and  the expansion $(1-t)^n=\sum_{q=0}^n(-1)^q\binom{n}{q}t^q$ in \eqref{eq:p-poly}, we can express $p_{i+d}$,  the coefficient of  $t^{i+d}$ in $P(t)$,
 as 
\[
p_{i+d}
=
\sum_{k=0}^{d-1}(-1)^{i+d-k}\binom{n}{k}\binom{n}{i+d-k}
\;+\; \Bigg(\binom{n}{r}-1 \Bigg)(-1)^i\binom{n}{i}.
\]
Reindex the above sum with $k=d-j$. Note that $\binom{n}{d-j}=\binom{n}{r+j}$ and
$\binom{n}{i+d-(d-j)}=\binom{n}{i+j}$.
We obtain
\[
(-1)^i p_{i+d}
=
\Bigg(\binom{n}{r}-1 \Bigg)\binom{n}{i}
+\sum_{j=1}^{n-r}(-1)^j\binom{n}{r+j}\binom{n}{i+j}
\]
which is exactly the $B_i$ from the statement of the proposition.

If $i=0$, then $\beta_{0,d}=0$ for any quotient $S/J$.

If $d-1\le i\le n$, it follows from the shape of the Betti table that $p_{i+d}=(-1)^i\beta_{i,i+d}$. Thus, we have 
\[
\beta_{i,i+d}=(-1)^i p_{i+d}=B_i,
\]
which gives $(a)$.

Finally, let $1\le i\le d-2$. If $i+d$ is odd, there is no diagonal term $\beta_{u,2u}$ to contribute to $p_{i+d}$. So, again $p_{i+d}=(-1)^i\beta_{i,i+d}$ and $\beta_{i,i+d}=B_i$.
If $i+d$ is even and $i\le d-4$, write $i+d=2u$ with $u=\frac{i+d}{2}\le d-2$. Then
\[
p_{i+d}=(-1)^i\beta_{i,i+d}+(-1)^u\binom{n}{u}.
\]
Multiplying by $(-1)^i$ gives
\[
\beta_{i,i+d}=(-1)^i p_{i+d}-(-1)^{i+u}\binom{n}{u}.
\]
Since $a=\frac{d-i}{2}$ is an integer with $a\geq 2$, $(-1)^{i+u}=(-1)^a$ and $\binom{n}{u}=\binom{n}{d-a}=\binom{n}{r+a}$.
Thus
\[
\beta_{i,i+d}=B_i-(-1)^a\binom{n}{r+a},
\]
proving $(b)$. This completes the proof.
\end{proof}

In the remaining part of this section, we focus on expressing the Betti numbers of $S/I_\M$  in terms of those of $S/J$.

\begin{theorem}\label{thm:betti-linkage}  
Fix integers $2 \le r < n$.  Let $\M=\U_{r,n}$ and $J=D:I_\M$ as above. Then for all $i,j$ one has
\[
\beta_{i,j}(S/I_\M)=
\begin{cases}
\binom{n}{i}, & \text{if } j = 2i \text{ and } j < i + r,\\
\beta_{n-r+2,\,2n-2r+2}(S/J)+\binom{n}{r-1}, & \text{if } (i,j) = (r-1,2r-2),\\
\beta_{n-r,2n-2r}(S/J) - \binom{n}{r}, & \text{if } (i,j) = (r+1,2r),\\
\beta_{n-i+1,\,2n-j}(S/J), & \text{if } j - i = r-1 \text{ and } i \neq r-1, r+1,\\
\beta_{1,n-r-1}(S/J) = 1, & \text{if } (i,j) = (n,n+r),\\
0, & \text{otherwise.}
\end{cases}
\]
\end{theorem}

\begin{figure}[htb!]
\centering
\resizebox{\textwidth}{!}{%

\begin{tikzpicture}[every node/.style={font=\small}]

\def\w{1.4}
\def\h{1.1}

\def\C{12}
\def\R{11}

\newcommand{\cellcenter}[2]{({(#1+0.5)*\w},{-(#2+0.5)*\h})}
\newcommand{\cellrect}[2]{({#1*\w},{-(#2)*\h}) rectangle ({(#1+1)*\w},{-(#2+1)*\h})}


\foreach \c in {1,2,3,4,5,6,7,8,9,10}{%
  \fill[gray!15] \cellrect{\c}{4};
}

\fill[gray!15] \cellrect{7}{5};
\fill[gray!15] \cellrect{8}{6};
\fill[gray!15] \cellrect{9}{7};
\fill[gray!15] \cellrect{10}{8};
\fill[gray!15] \cellrect{11}{9};
\fill[gray!15] \cellrect{10}{5};

\draw[thick] (0,0) rectangle (\C*\w,-\R*\h);
\foreach \c in {1,...,11} { \draw (\c*\w,0) -- (\c*\w,-\R*\h); }
\foreach \r in {1,...,10} { \draw (0,-\r*\h) -- (\C*\w,-\r*\h); }

\foreach \c/\lab in {
  0/0,
  1/1,
  2/2,
  3/$\cdots$,
  4/$r\!-\!1$,
  5/$r$,
  6/$r\!+\!1$,
  7/$r\!+\!2$,
  8/$\cdots$,
  9/$n\!-\!1$,
  10/$n$,
  11/$n\!+\!1$
}{
  \node at ({(\c+0.5)*\w}, {0.55*\h}) {\lab};
}

\foreach \r/\lab in {
  0/0,
  1/1,
  2/2,
  3/$\vdots$,
  4/$r\!-\!1$,
  5/$r$,
  6/$r\!+\!1$,
  7/$\vdots$,
  8/$n\!-\!2$,
  9/$n\!-\!1$,
  10/$n$
}{
  \node[anchor=east] at ({-0.25*\w}, {-(\r+0.5)*\h}) {\lab};
}


\node at \cellcenter{0}{0} {$1$};
\node at \cellcenter{1}{1} {$\binom{n}{1}$};
\node at \cellcenter{2}{2} {$\binom{n}{2}$};
\node at \cellcenter{4}{4} {};
\node at \cellcenter{5}{5} {$\binom{n}{r}$};
\node at \cellcenter{6}{6} {$\binom{n}{r+1}$};
\node at \cellcenter{8}{8} {$\binom{n}{n-2}$};
\node at \cellcenter{9}{9} {$\binom{n}{n-1}$};
\node at \cellcenter{10}{10} {$\binom{n}{n}$};

\node at \cellcenter{11}{9} {$1$};
\node at \cellcenter{10}{8} {$\binom{n}{1}$};
\node at \cellcenter{7}{5} {$\binom{n}{n-r-1}$};

\node at \cellcenter{9}{7} {$\ddots$};
\node at \cellcenter{8}{6} {$\ddots$};
\node at \cellcenter{7}{7} {$\ddots$};
\node at \cellcenter{3}{3} {$\ddots$};

\node at \cellcenter{10}{5} {$1$};

\node at \cellcenter{10}{4} {$\binom{n}{r-1}\!-\!n$};
\node at \cellcenter{8}{4} {$\cdots$};
\node at \cellcenter{3}{4} {$\cdots$};
\node at \cellcenter{1}{4} {$\binom{n}{r}-1$};

\end{tikzpicture}
}
\caption{Mapping cone construction $G_\bullet$}
\label{fig:bettiTable2}
\end{figure}

\begin{proof} Let $F_\bullet$ denote the minimal free resolution of $S/J$, and $K_\bullet$ the minimal free resolution of $D$. Then, $F_\bullet^\ast(-n)$ resolves $\omega_{S/J}(-n)$; hence 
\[
\beta_{i,j}(\omega_{S/J}(-n)) = \beta_{n-i+1,2n-j}(S/J).
\] 
Let $\varphi_\bullet$ be the lift of the canonical injection $\omega_{S/J}(-n) \cong I_\M/D \to S/D$. Now, $S/I_\M$ is resolved non-minimally by the mapping cone of $F_\bullet^\ast \xrightarrow{\varphi_\bullet} K_\bullet$, which we denote by $G_\bullet$. We depict $G_\bullet$ in \Cref{fig:bettiTable2}, and we shade the cells in gray where $F_\bullet^\ast$ is concentrated.

Using \Cref{HS1}, we observe $\reg_S(S/I_\M) = r$ and so, for $j>i+r$,
\[
\beta_{i,j}(S/I_\M) = 0.
\]
Thus, there must be consecutive cancellations in $G_\bullet$ of the term $K_i$ with $F_{i+1}^\ast$ for $i \geq r+1$.  Since $S/I_\M$ is Cohen-Macaulay, the maximal graded shifts are strictly increasing \cite[Proposition 2.2]{M21}, and therefore there must be consecutive cancellations in $G_\bullet$ of the term $K_r$ with $F_{r+1}^\ast$. To summarize, we have shown that $\beta_{i,i+r}(S/I_\M) = 0$ for $i < n$. After performing these consecutive cancellations, the structure of the resolution is such that no more consecutive cancellations are possible, and thus is the minimal free resolution of $S/I_\M$. The formula for the Betti numbers now follows by recording the remaining Betti numbers.
\end{proof}

\begin{example}
    Take $\M = \U_{6,8}$.  Then the Betti table of $S/I_\M$ is depicted in \Cref{fig:Betti}.
     As verification of the formulas in \Cref{thm:betti-linkage} and \Cref{prop:betti-from-HS}, note the following

     \begin{enumerate}
         \item $\beta_{5,10}(S/I_\M) = \beta_{4,6}(S/J) + \binom{8}{5} = B_4 + \binom{8}{5} = 1470+56 = 1526.$
         \item $\beta_{6,11}(S/I_\M) = \beta_{3,5}(S/J) = B_3 = 1008$.
         \item $\beta_{7,12}(S/I_\M) = \beta_{2,4}(S/J) - \binom{8}{6} = B_2 - \binom{8}{6} = 378-28 = 350$.
     \end{enumerate}

\begin{figure}[htb!]
    $$\begin{tabular}{|r|ccccccccc|}
    \hline
        & 0 & 1 & 2 & 3 & 4 & 5 & 6 & 7 & 8\\
\hline
     0: & 1 & \text{-} & \text{-} & \text{-} & \text{-} & \text{-} & \text{-} & \text{-} & \text{-}\\
     1: & \text{-} & 8 & \text{-} & \text{-} & \text{-} & \text{-} & \text{-} & \text{-} & \text{-}\\
     2: & \text{-} & \text{-} & 28 & \text{-} & \text{-} & \text{-} & \text{-} & \text{-} & \text{-}\\
     3: & \text{-} & \text{-} & \text{-} & 56 & \text{-} & \text{-} & \text{-} & \text{-} & \text{-}\\
     4: & \text{-} & \text{-} & \text{-} & \text{-} & 70 & \text{-} & \text{-} & \text{-} & \text{-}\\
     5: & \text{-} & 27 & 208 & 693 & 1296 & 1526 & 1008 & 350 & 48\\
     6: & \text{-} & \text{-} & \text{-} & \text{-} & \text{-} & \text{-} & \text{-} & \text{-} & 1\\
     \hline
     \end{tabular}$$
     \caption{Betti table for $S/I_{\U_{6,8}}$.}\label{fig:Betti}
\end{figure}

\end{example}

\section{Linearly presented graded M\"obius algebras}\label{linpres}

In this section, we consider consider the question of when the defining ideal of the graded M\"obius algebra of a simple matroid is both quadratic and linearly presented.  Recall that a matroid is called \newWord{Sylvester} if every pair of elements belongs to a 3-element circuit~\cite{Murty70}.

Computing many examples have led us to the following conjecture.

\begin{conjecture}\label{conj:Sylv_iff_lin_pres}
    Let $\M$ be a simple matroid with at least $3$ elements such that $\GMA{\M}$ is quadratic. Then $\GMA{\M}$ is linearly presented   if and only if  $\M$ is Sylvester.
\end{conjecture}

\noindent The uniform matroids $\U_{2,n}$ for $n \ge 3$ are quadratic and Sylvester.  It follows from \Cref{thm:betti-linkage} that $I_{\U_{2,n}}$ are quadratic and linearly presented.  Thus the conjecture holds for these uniform matroids.  One can check with Macaulay2 that small affine and projective geometries, like the Fano matroid, also satisfy the conjecture.  

The following examples show that neither the quadratic nor Sylvester condition alone is sufficient for $\GMA{\M}$ to be both quadratic and linearly presented.

\begin{figure}[htb!]
    \begin{center}
$\begin{tabular}{r|ccccccc}
       & 0 & 1 & 2 & 3 & 4 & 5 & 6\\
       \hline
      0: & 1 & \text{-} & \text{-} & \text{-} & \text{-} & \text{-} & \text{-}\\
      1: & \text{-} & 14 & 29 & 9 & \text{-} & \text{-} & \text{-}\\
      2: & \text{-} & \text{-} & 3 & 41 & 50 & 21 & 1\\
      3: & \text{-} & \text{-} & \text{-} & \text{-} & \text{-} & \text{-} & 1
      \end{tabular}$
    \end{center}
\caption{Betti table for $\GMA{\M(K_4)}$}\label{fig:GMA:K4}
    \end{figure}

\begin{example}
    Let $\M = \M(K_4)$ be the graphic matroid of the complete graph $K_4$ on $4$ vertices.  By \cite[Theorem 5.1]{LMMP24}, 
    $\GMA{\M}$ is quadratic.  However, there are no 3-circuits containing disjoint edges, meaning $\M$ is not Sylvester.  The Betti table in Figure~\ref{fig:GMA:K4} shows that the corresponding graded M\"obius algebra is not linearly presented.
\end{example}

    In fact, the only nontrivial graph whose graphic matroid is Sylvester is the cycle graph $C_3$.  It's graded M\"obius algebra is quadratic and linearly presented.

\begin{example}\label{ex:truncation}
    Let $\M = \mathbb{F}_2^4 \smallsetminus \{\underline{0}\} = \mathbb{P}^3_{\mathbb{F}_2}$.  
    Set $\mathsf{T} = T(\M)$ to be the truncation of $\M$.  Given any two elements $a,b $ of  $\mathsf{T}$, then $\{a,b,a+b\}$ is a circuit of size $3$; thus, $\mathsf{T}$ is Sylvester.  A computation with Macaulay2 shows that $\GMA{\mathsf{T}}$ is not quadratic.
\end{example}

We prove the forward implication of \Cref{conj:Sylv_iff_lin_pres}.
\begin{proposition}\label{prop:linpres:Sylvester}
    If $\M$ is a simple matroid, and $\GMA{\M}$ is quadratic and linearly presented, then $\M$ is Sylvester.
\end{proposition}

\begin{proof}
Suppose toward a contradiction that $\M$ is not Sylvester. Then there exist distinct elements $a$, $b$ of $E$ which do not belong to a common $3$-circuit. Thus, $y_a y_b$ does not appear in the monomial support of any of the minimal generators of $[I_\M]_2$. Set $A = \GMA{\M}$, and consider the Koszul complex $K_\bullet(\underline{y};A) := A \otimes_{\KK} \bigwedge {\KK}^n$ on the variables of $S$. Consider the element 
\begin{align*}
    \alpha := y_a y_b \otimes e_a \wedge e_b \in [A]_2 \otimes_{\KK} \bigwedge^2 {\KK}^n = [K_2]_4.
\end{align*}
We observe that $\alpha$ is a cycle. We prove below the claim that $\alpha$ is not a boundary in $K_\bullet(\underline{y};A)$. In which case, $\beta_{2,4}(A) \neq 0$. This contradicts $I_\M$ being both quadratic and linearly presented.

\underline{Claim}: $\alpha$ is not a boundary of $K_\bullet(\underline{y};A)$.

If $\alpha$ was a boundary, then there would exist $f \in [A \otimes \bigwedge^3 S^n]_4 = [A]_1 \otimes \bigwedge^3 S^n$ such that $\partial(f) = \alpha$. If $e_a \wedge e_b$ appears as a  term in $\partial(f)$, then $f$ must have the form   $$f := \sum_{\substack{u,c \in E\\ c \neq a,b}} \lambda_{u,c} y_u \otimes e_a \wedge e_b \wedge e_c + g$$ with $\lambda_{u,c} \in \KK$ for all $u,c$ and the terms of $g$ not divisible by $e_a \wedge e_b$. Applying $\partial$ to $f$, those elements containing $e_a \wedge e_b$ are of the form
\begin{align*}
    \sum_{u,c} (-1)^{\varepsilon(a,b,c)} \lambda_{u,c} y_uy_c \otimes e_a \wedge e_b,
\end{align*}
for some function $\varepsilon(a,b,c)$. Since $\partial(f) = \alpha$, it must be the case that 
\[
\sum_{u,c} (-1)^{\varepsilon(a,b,c)} \lambda_{u,c} y_uy_c = y_a y_b.
\] 
Since $c$ is not equal to $a$ nor $b$, this yields a non-trivial quadratic relation involving $y_ay_b$ in $I_\M$, which we observed previously cannot happen.
\end{proof}

The previous result has implications for the following open question:

\begin{question}[{\cite[Question 7.10]{McCullough21}}] \label{qst:lin:pres}
    Suppose $I_n \subseteq \KK[x_1,\ldots,x_n]$ be a family of graded, quadratic, linearly presented ideals.  Is
    \[\limsup \frac{\reg(I_n)}{\sqrt{n}} < \infty?  \]
\end{question}

The answer is affirmative for quadratic monomial ideals and defining ideals of Veronese and Grassmannian varieties, but the general case remains open.  The previous result allows us to show that the answer is also affirmative for defining ideals of graded M\"obius algebras.

\begin{proposition}\label{prop:lin:pres:reg}
    Let $\M_n$ be a sequence of matroids such that the underlying set has size $|E(\M_n)| = n$ and such that $I_{\M_n}$ is quadratic and linearly presented for all $n$.  Then 
    \[\lim_{n \to \infty} \frac{\reg(I_{\M_n})}{\sqrt{n}} = 0.\]
    In particular, the answer to \Cref{qst:lin:pres} is affirmative in this case.  
\end{proposition}

\begin{proof}
    By \Cref{prop:linpres:Sylvester}, $\M_n$ is Sylvester for all $n$.  Let $r_n = \rk(\M_n)$.    Then $\reg(I_n)= r_n+1$. We may assume that $\limsup r_n = \infty$, since otherwise the numerator is bounded while the denominator tends to infinity. By \cite[Theorem 1]{Murty70}, since $\M_n$ is Sylvester, $n = |E(\M_n)| \ge 2^{r_n}-1$ for all $n$.  Therefore
    \[ 0 \le \lim_{n \to \infty} \frac{\reg(I_{\M_n})}{\sqrt{n}} \le \lim_{r \to \infty} \frac{r+1}{\sqrt{2^r - 1}} = 0. \qedhere
    \]
\end{proof}


\section*{Acknowledgements}

The authors would like to thank the organizers Louiza Fouli, Jonathan Monta\~no, and Michael DiPasquale for organizing the conference and workshop ``Arizona-New Mexico Symposium" where this work was initiated. This conference was supported by the National Science Foundation grants DMS–2401522 and DMS–2344588, the Department of Mathematics at New Mexico State University, and the School of Mathematics and Statistics at Arizona State University.  The authors also thank Chayim Lowen for helpful feedback on an earlier draft of this paper.

Computations in Macaulay2 \cite{M2} were especially  helpful in writing this paper.
ChatGPT insights were used in the finding \Cref{ex:truncation} and proving \Cref{prop:linpres:Sylvester}.  

Kara was partially supported by NSF grant DMS--2418805. LaClair was partially supported by NSF grant DMS-2342256.  Grate and McCullough were partially supported by NSF grant  DMS--2401256. Sobieska was partially supported by NSF grant DMS--2532902.

\bibliographystyle{amsalpha}
\bibliography{ref}

\newcommand{\etalchar}[1]{$^{#1}$}
\providecommand{\bysame}{\leavevmode\hbox to3em{\hrulefill}\thinspace}
\providecommand{\MR}{\relax\ifhmode\unskip\space\fi MR }
\providecommand{\MRhref}[2]{%
  \href{http://www.ams.org/mathscinet-getitem?mr=#1}{#2}
}
\providecommand{\href}[2]{#2}
\begin{thebibliography}{BHM{\etalchar{+}}22}

\bibitem[BHM{\etalchar{+}}20]{singular:Hodge:theory:for:combinatorial:geometries}
Tom Braden, June Huh, Jacob~P. Matherne, Nicholas Proudfoot, and Botong Wang,
  \emph{Singular {H}odge theory for combinatorial geometries}, 2020, arXiv.CO
  2010.06088, to appear in Journal of the American Mathematical Society.

\bibitem[BHM{\etalchar{+}}22]{semi-small:decompositions}
Tom Braden, June Huh, Jacob~P. Matherne, Nicholas Proudfoot, and Botong Wang,
  \emph{A semi-small decomposition of the {C}how ring of a matroid}, Adv. Math.
  \textbf{409} (2022), no.~part A, Paper No. 108646, 49. \MR{4477425}

\bibitem[GS]{M2}
Daniel~R. Grayson and Michael~E. Stillman, \emph{Macaulay2, a software system
  for research in algebraic geometry}, Available at
  \url{http://www2.macaulay2.com}.

\bibitem[LMMP25]{LMMP24}
Adam LaClair, Matthew Mastroeni, Jason McCullough, and Irena Peeva,
  \emph{Koszul graded {M}\"obius algebras and strongly chordal graphs}, Selecta
  Math. (N.S.) \textbf{31} (2025), no.~2, Paper No. 29, 30. \MR{4875052}

\bibitem[McC21a]{M21}
Jason McCullough, \emph{On the maximal graded shifts of ideals and modules}, J.
  Algebra \textbf{571} (2021), 121--133. \MR{4200712}

\bibitem[McC21b]{McCullough21}
\bysame, \emph{Subadditivity of syzygies of ideals and related problems},
  Commutative algebra, Springer, Cham, 2021, pp.~501--522. \MR{4394419}

\bibitem[MN16]{Maeno:Numata:published}
Toshiaki Maeno and Yasuhide Numata, \emph{Sperner property and
  finite-dimensional {G}orenstein algebras associated to matroids}, J. Commut.
  Algebra \textbf{8} (2016), no.~4, 549--570. \MR{3566530}

\bibitem[Mur70]{Murty70}
U.S.R. Murty, \emph{Matroids with {S}ylvester property.}, Aequationes
  mathematicae \textbf{4} (1970), 44--50 (eng).

\bibitem[Oxl11]{Oxley}
James Oxley, \emph{Matroid theory}, second ed., Oxford Graduate Texts in
  Mathematics, vol.~21, Oxford University Press, Oxford, 2011. \MR{2849819}

\bibitem[Pee11]{graded:syzygies}
Irena Peeva, \emph{Graded syzygies}, Algebra and Applications, vol.~14,
  Springer-Verlag London, Ltd., London, 2011. \MR{2560561}

\bibitem[Yuz01]{yuz01}
S~A Yuzvinsky, \emph{Orlik-{S}olomon algebras in algebra and topology}, Russian
  Mathematical Surveys \textbf{56} (2001), no.~2, 293.

\end{thebibliography}
\end{document}